\documentclass[11pt]{amsart}

\usepackage{amssymb}
\usepackage[colorlinks=true, urlcolor=rltblue, citecolor=drkgreen, linkcolor=drkred] {hyperref}
\usepackage{color}
\usepackage{pdfsync}
\definecolor{rltblue}{rgb}{0,0,0.4}
\definecolor{drkgreen}{rgb}{0,0.4,0}
\definecolor{drkred}{rgb}{0.5,0,0}

\newtheorem{thm}{Theorem}[section]
\newtheorem{lemma}[thm]{Lemma}

\newtheorem{theorem}[thm]{Theorem}
\newtheorem{corollary}[thm]{Corollary}

\theoremstyle{definition}
\newtheorem{definition}[thm]{Definition}

\theoremstyle{remark}

\newtheorem{historic}[thm]{Historic Remark}

\theoremstyle{plain}

\newcounter{contenumi}

\def\leqt{\leq_T}
\def\geqt{\geq_T}

\def\teq{\equiv_T}

\def\upto{\mathop{\upharpoonright}}

\def\and{\mathrel{\&}}

\def\isom{\cong}
\def\Si{\Sigma}

\newcommand\rightdate[1]{\footnotetext{  Saved: #1 \\ Compiled: \today}}
\def\A{\mathcal{A}}
\def\B{\mathcal{B}}
\def\C{\mathcal{C}}

\def\om{\omega}
\def\bbar{\bar{b}}

\def\si{\sigma}
\def\b{\beta}

 \newcommand{\BFI}[1]{\ensuremath{ \textbf{bf}_{#1}}} 
 
\def\a{\alpha}
\def\b{\beta}

\def\ZZ{{\mathbb Z}}
\def\QQ{{\mathbb Q}}
\def\M{{\mathcal M}}

\def\A{\mathcal A}

\def\L{\mathcal L}
\def\H{\mathcal H}
\def\xbar{\bar{x}}
\def\cbar{\bar{c}}
\def\ybar{\bar{y}}
\def\abar{\bar{a}}
\def\bbar{\bar{b}}
\def\xbar{\bar{x}}
\def\dbar{\bar{d}}
\def\ctt{{\mathtt c}}
\def\itt{{\mathtt {in}}}
\def\Sic{\Si^\ctt}
\def\Pic{\Pi^\ctt}
\newcommand{\Sico}[1]{\Si^{\ctt,#1}}
\newcommand{\Pico}[1]{\Pi^{\ctt,#1}}
\def\Sii{\Si^\itt}
\def\Pii{\Pi^\itt}

\def\Cfr{{\mathfrak C}}
\def\KK{{\mathbb K}}

\def\BB{{\mathbb B}}

\def\ext{ext}

\def\tp{\text{-}tp}

\def\g{\gamma}
\def\d{\delta}

\def\fin{{\textit{fin}}}

\def\Si{\Sigma}

\def\Sp {\mathop{\mathrm{Sp}}}
\def\om{\omega}

\def\Sp{Sp}

\def\implies{\Rightarrow}

\def\Bb{\mathbb{B}}
\def\Cc{\mathbb{C}}

\def\Cfr{\mathfrak{C}}
\def\Kfr{\mathfrak{K}}

\def\bfphi{\boldsymbol{\varphi}}
\def\bfpsi{\boldsymbol{\psi}}

\title{A computability theoretic equivalent to Vaught's conjecture}

\author{Antonio Montalb\'an}
\thanks{The author was partially supported by NSF grant DMS-0901169 and the Packard Fellowship.
The author would also like to thank Asher Kach for useful comments on an earlier draft of this paper.}

\address{Department of Mathematics\\
University of Chicago\\
5734 S. University ave.\\
Chicago, IL 60637, USA}

\email{antonio@math.uchicago.edu}
\urladdr{\href{http://www.math.uchicago.edu/~antonio/index.html}{www.math.uchicago.edu/$\sim$antonio}}

\begin{document}

\rightdate{August 22, 2012 (Submitted version but with new title)}
\maketitle


%

\begin{abstract}
We prove that, for every theory $T$ which is given by an ${\mathcal L}_{\omega_1,\omega}$ sentence, $T$ has less than $2^{\aleph_0}$ many countable models if and only if we have that, for every $X\in 2^\omega$ on a cone of Turing degrees, every $X$-hyperarithmetic model of $T$ has an $X$-computable copy.
We also find a concrete description, relative to some oracle, of the Turing-degree spectra of all the models of a counterexample to Vaught's conjecture.
\end{abstract}

\section{Introduction}

The main result of this paper shows that Vaught's conjecture can be viewed as a computability theoretic question. 
Vaught's conjecture states that the number of countable models of a first-order theory is either countable or continuum.
It was conjectured in 1961 \cite{Vau61}, and is one of the longest-standing, best-known, open questions in mathematical logic.
Though phrased as a  question in model theory,  it is related to other areas of logic too.
For instance, it can be viewed as a descriptive set theoretic question; the topological Vaught's conjecture is a well-known strengthening of Vaught's conjecture which is  purely descriptive set theoretic in nature.
The idea that it is also connected to notions from  higher computability theory has been around for some time too.
Gerald Sacks has done much work on this approach (see, for instance, \cite{Sac83, Sac07}).
In this paper we give a precise computability theoretic formulation of Vaught's conjecture.

We  state our main theorem, describing what each of the items means afterwards. 

\begin{theorem} (ZFC + PD)  \label{thm: main}
Let $T$ be a $\L_{\om_1,\om}$-sentence with uncountably many countable models.
The following are equivalent:
\begin{enumerate}\renewcommand{\theenumi}{V{\arabic{enumi}}}
\item $T$ is a counterexample to Vaught's conjecture.   \label{main: part 1}
\item $T$ satisfies hyperarithmetic-is-recursive on a cone.     \label{main: part 2}
\item There exists an oracle relative to which,     \label{minimal main: part 2} \label{main: part 3}
\[
\{\Sp(\A): \A\models T\}\  =\  \{\{X\in 2^\om: \om^X_1\geq \a\}: \a \in\om_1 \}.
\]
\end{enumerate}
\end{theorem}

The theorem is stated not only for finitary first-order theories, but for $\L_{\om_1,\om}$ sentences, where $\L_{\om_1,\om}$ is the language where countable infinitary disjunctions and infinitary conjunctions are allowed.
In this paper, when we refer to Vaught's conjecture, we will refer to Vaught's conjecture for $\L_{\om_1,\om}$ sentences.
Given the techniques we use here, our results apply to both infinitary or finitary theories.

On (\ref{main: part 1}), when the continuum hypothesis (CH) does not hold, by a {\em counterexample to Vaught's conjecture} we mean an $\L_{\om_1,\om}$ sentence $T$ which has uncountably many countable models but not continuum many.
In Definition \ref{def: scattered} below we will define a counterexample to Vaught's conjecture to be a scattered theory with uncountably many models, which is equivalent to the definition above under $\neg$CH, and still makes sense with or without CH.

To describe (\ref{main: part 2}), we need the following new definition.

\begin{definition}
We say that a class of structures $\KK$ satisfies {\em hyperarithmetic-is-recursive} if every hyperarithmetic structure in $\KK$ has a computable copy.
\end{definition}

This is true for the class of well-orders (as proved by Spector \cite{Spe55}).
The other classes we know, as for example the class of super-atomic Boolean algebras, are, in some sense, imitations of the class of well-orders.
None of these examples is  axiomatizable by an $\L_{\om_1,\om}$ sentence.
However, we can get other examples if we replace isomorphism by other equivalence relations.
For instance, the author \cite{MonEqui} showed  that every hyperarithmetic linear ordering is bi-embeddable with a computable one.
Greenberg and the author \cite{GMRanked} showed the same for $p$-groups.
In Section \ref{se: hyp-is-rec} we give a general theorem that encapsulates all these examples, but that only works relative to a cone.
We say that $\KK$ satisfies {\em hyperarithmetic-is-recursive on a cone} if there exists $Y$ such that for every $X\geqt Y$, every $X$-hyperarithmetic structure in $\KK$ has an $X$-computable copy.

The third item, (\ref{main: part 3}), fully describes the degree-spectrum of all the models of $T$, relative to an oracle.
The spectrum of a structure is widely used as a way of measuring the computational complexity of a structure, and much research has been done on it.
The {\em spectrum of a structure $\A$}, $\Sp(\A)$, is the set of all $X\in 2^\om$ which can compute a copy of $\A$.
Given $X\in 2^\om$, $\om_1^X$ is the least ordinal which has no copy computable in $X$.
So, for instance, if the structure $\A$ is a well-ordering of order type $\b$, then $\Sp(\A)=\{X\in 2^\om: \om_1^X > \b\}$.
It is easy to see that (\ref{main: part 3}) implies (\ref{main: part 2}).

We remark that we do know $\Si^1_1$ classes $\KK$ of structures with 
\[
\{\Sp(\A): \A\in \KK\} = \{\{X\in 2^\om: \om^X_1\geq \a\}: \a \in\om_1 \},
\]
such as, for instance, the class of linear orderings $\KK= \{\ZZ^\a\cdot \QQ: \a\in\om_1\}$.
 (This is Kunen's example, see \cite[1.4.3]{Ste78}.)
On the other hand, the class of ordinals does not have this family of spectra because if  $\a$ is both admissible and a limit of admissibles, then no ordinal has spectrum $\{X\in 2^\om: \om^X_1\geq \a\}$.
We  do not know of any $\Pi^1_1$ class of structures with exactly that family of spectra.

We also note that, essentially, the only structures we know with spectrum $\{X:\om_1^X>\om_1^{CK}\}$ are $\om_1^{CK}$ itself, all ordinals in between $\om_1^{CK}$ and $\om_2^{CK}$, and the linear ordering $\om_2^{CK}\cdot(1+\QQ)$.
All the other examples we know are produced out of these in one way or another.

Theorem \ref{thm: main} is not proved within ZFC, since some amount of determinacy is needed in the proof.
Projective determinacy (PD) is enough for our proofs, but much less should be enough too.
Projective determinacy is the statement that every projective set is determined, and it follows from the existence of infinitely many Woodin cardinals (Martin and Steel \cite{MS89}).
We do not know whether the results of this paper can be obtained in ZFC.

One of the main tools applied in this paper is Martin's turing determinacy, that says that every Turing-degree-invariant set of reals either contains a cone or is disjoint from a cone of Truing degrees.
This is where  projective determinacy is needed, and the oracle relative to which parts (\ref{main: part 2}) and (\ref{main: part 3}) are proved come from strategies in projective games.
(See Section \ref{se: martins} for more background information.)

An important notion used in this paper is the {\em back-and-forth structure} of a class of structures.
The careful study of this structure was started by the author in \cite{McountingBF}.
We review all the necessary background in Section \ref{se: bf}.
It is worth mentioning that these techniques are not necessary if the objective of the reader is to prove Vaught's conjecture.
If one managed to prove that no theory can have a class of spectra as in (\ref{main: part 3}), Vaught's conjecture would then follow using only the results in Section \ref{se: hyp-is-rec}.

This paper is organized as follows.
In Section \ref{ss: background} we quickly review the basic background necessary for the rest of the paper.
We start Section \ref{se: hyp-is-rec} by reviewing Martin's theorems and then we use them  to prove a general result (Theorem \ref{thm: easy version}) saying that certain equivalence relations satisfy hyperarithmetic-is-recursive.
In Section \ref{se: minimal} we show that if $T$ is a minimal counterexample to Vaught's conjecture, there exists an oracle relative to which, all models $\A$ of $T$ satisfy that $\Sp(\A)=\{Z\in 2^\om: \om_1^\A\leq\om_1^Z\}$ (Theorem \ref{thm: minimal example}).
It will then easily follow that all such $T$ satisfy (\ref{main: part 3}).
To prove Theorem \ref{thm: minimal example} we will develop some of the lemmas that we will need later.
We prove Theorem \ref{thm: minimal example} separate from our main theorem because its proof does not need the material in Sections \ref{se: bf} and \ref{se: main}, and only uses Theorem \ref{thm: easy version} together with known results from higher recursion theory.
Thus, if what the reader wants is to use our results to prove Vaught's conjecture, there is no need to go through Sections \ref{se: bf} and \ref{se: main}.
In Section \ref{se: bf} we develop the $\a$-back-and-forth structures, extending ideas from the author's previous work in \cite{McountingBF}.
These structures provide an important tool to prove our main results in Section \ref{se: main}.

\subsection{Background and Notation}  \label{ss: background}

We  deal only with countable structures.
For now on, by structure, we mean countable structure.

\subsubsection{Infinitary languages}

We will use the infinitary language $\L_{\om_1,\om}$ and its effective version $\L^c_{\om_1,\om}$ throughout this paper.
We refer the reader to \cite[Chapters 6 and 7]{AK00} for background on infinitary formulas.
Without loss of generality, we will assume $\L$ is a relational computable language.
An $\L_{\om_1,\om}$ formula is {\em computable} if all the infinite disjunctions and conjunctions are  over computably enumerable sets of formulas.
When we count alternations of quantifiers, we now count infinitary disjunctions as existential quantifiers and infinitary conjunctions as universal quantifiers.
This way we get hierarchies $\Sii_\a$ and $\Pii_\a$ of infinitary formulas with $\a$ ranging in $\om_1$, and $\Sic_\a$ and $\Pic_\a$ of computable infinitary formulas with $\a$ ranging in $\om_1^{CK}$.
We use $\L^c_{\om_1,\om}$ for the whole set of computably infinitary formulas, and $\Sic_{<\a}$ for $\bigcup_{\b<\a}\Sic_\b$.
At times we will consider computable infinitary formulas relative to some oracle $X$, in which case we allow the infinite disjunctions and conjunctions to be $X$-computably enumerable.
We denote the relativized hierarchies of formulas by $\Sico{X}_\a$, $\Pico{X}_\a$, and $\L^{c,X}_{\om_1,\om}$.
Note that every $\Sii_\a$ formula is $\Sico{X}_\a$ for some $X$.

\subsubsection{Scott Rank}

Fix a structure $\A$.
Given a tuple $\abar$ in $\A$, we define $\rho_\A(\abar)$ to be the least ordinal $\a$ such that if all $\Pii_\a$ formulas true of $\abar$ are true of another tuple $\bbar$, then all $\L_{\om_1,\om}$ formulas true of $\abar$ are true of $\bbar$.
The {\em Scott rank} of $\A$ is then given by $SR(\A)= \sup\{\rho(\abar)+1: \abar\in\A\}$.
This function $SR$  is called $R$ in \cite[Section 6.7]{AK00}.

Scott's isomorphism theorem \cite{Sco65} then states that there is a $\Pii_{SR(\A)+2}$-sentence which is true in $\A$ but not in any other countable structure. 
In particular,  if two countable structures are $\L_{\om_1,\om}$-elementary equivalent,  they are isomorphic.
The effective version of this latter statement holds as well: 

\begin{theorem}[{Nadel \cite{Nad74} (see also \cite[Theorem 7.3]{Bar75})}]\label{thm: effective Scott}
If two $X$-computable structures are $\L^{c,X}_{\om_1,\om}$-elementary equivalent, they are isomorphic.
\end{theorem}

However, there are computable structures for which no $\L^c_{\om_1,\om}$ sentence determines their isomorphism type among countable, or even computable, structures.

\subsubsection{Morley's theorem} \label{se: morely}

One of the most important partial results towards Vaught's conjecture is Morley's theorem \cite{Mor70}, which states that the number of countable models of an $\L_{\om_1,\om}$ sentence is either countable, $\aleph_1$, or $2^{\aleph_0}$.
This result can be proved as follows.
Given two structures, we let $\A\equiv_\a\B$ if $\A$ and $\B$ are $\Sii_\a$-elementary equivalent (see Section \ref{se: bf} for more on this relation).
For each $\a<\om_1$, this equivalence is Borel (using Definition \ref{thm: def bnf}), and hence, by Silver's theorem \cite{Sil80}, the number of $\equiv_\a$-equivalence classes among the models of a theory $T$ is either countable or continuum.
Therefore, if  the number of models of $T$ is less than $2^{\aleph_0}$, the number of $\equiv_\a$-equivalence classes must be countable for each $\a$.
By Scott's theorem, for each $\A$ and for $\a=SR(\A)+2$, we have that for any other structure $\B$, if $\B\equiv_\a\A$, then $\B\isom \A$.
Thus, for each given $\b$, the number of models of $T$ of Scott rank $\b$ is at most countable.
It follows that the number of models of $T$ has to be at most $\aleph_1$.

\begin{definition}\label{def: scattered}
A theory $T$ is {\em scattered} if  for each $\a<\om_1$, there are at most countably many $\equiv_\a$-equivalence classes among the models of $T$.
If also $T$ has uncountably many models, we say that $T$ is a {\em counterexample to Vaught's conjecture}.
\end{definition}

\subsubsection{Omega-one-Church-Kleene} \label{se: omega 1 CK}

The least ordinal with no computable presentation is denoted $\om_1^{CK}$.
Given $X\in 2^\om$, we let $\om_1^X$ be the least ordinal with no $X$-computable presentation, or equivalently, with no $X$-hyperarithmetic presentation.
Sacks showed that an ordinal $\a$ is {\em admissible} if and only if it is of the form $\om_1^X$ for some $X$.
The reader might use this as the definition of admissible ordinal throughout this paper.

It can be proved using Theorem \ref{thm: effective Scott} that for every $X$-computable structure $\A$, $SR(\A)\leq \om_1^X+1$, and there are examples where this maximum is attained, like the Harrison linear ordering which has order type $\om_1^X\cdot(1+\QQ)$.
Give a structure $\A$, we let 
\[
\om_1^\A = \min\{\om_1^X: X\in \Sp(\A)\}.
\]
So, we have that 
\[
SR(\A)\leq \om_1^\A+1.
\]
Structures for which $SR(\A)\geq \om_1^\A$ are said to have {\em high Scott rank}.
Most of the difficulty proving the results of this paper comes from dealing  with the structures of high Scott rank.
Sacks \cite{Sac07} showed that if a scattered theory has no models with $SR(\A)= \om_1^\A+1$, then it has countably many models.

We remark that the Scott rank of a structure is not affected by relativization, but $\om_1^\A$ is.
Given $Y\in 2^\om$, we let $\om_1^{\A,Y} =  \min\{\om_1^X: X\in \Sp(\A), X\geqt Y\}$.
Thus, the notion of ``high Soctt rank'' is not invariant under relativization.
We notice, however, that if $\A$ has high Scott rank relative to some $Y$, that is if $SR(\A)\geq \om_1^{\A, Y}$, then it has high Scott rank and $\om_1^{\A, Y}=\om_1^\A$, because $\om_1^{\A, Y}\leq SR(\A) \leq \om_1^\A+1 \leq \om_1^{\A,Y}+1$.

\section{Hyperarithmetic is recursive} \label{se: hyp-is-rec}

In this section we give a sufficient condition for an equivalence class to satisfy hyperarithmetic-is-recursive.
This condition is quite general and includes some interesting examples.
But it is stronger than the hypothesis of our main theorem, which will need a different, more elaborate proof.
The main technique used in this section is Martin's Turing determinacy.

We start with a general definition.

\begin{definition}
A {\em ranked equivalence relation} consist of a set $\Kfr\subseteq 2^\om$, an equivalence relation $\equiv$ on $\Kfr$, and a function $r\colon \Kfr/\equiv \to \om_1$.
We say that a ranked equivalence relation is {\em projective} if $\Kfr$ and $\equiv$ are  projective and $r$ has a presentation $2^\om\to 2^\om$ which is projective.
We say that $(\Kfr,\equiv,r)$ is {\em scattered} if $r^{-1}(\alpha)$ contains only countably many equivalence classes for each $\alpha\in\om_1$.
\end{definition}

The main example of a projective ranked equivalence relation is the isomorphism relation on a class of structures, where the rank function is the Scott rank.
When the class of structures is the class of models of a scattered theory $T$, we have that this ranked equivalence relation is scattered.
We will not be able to get that every scattered projective ranked equivalence relation satisfies hyperarithmetic-is-recursive on a cone, but we will get close.

We use $[X]$ to denote the equivalence class of $X$.

\begin{theorem}\label{thm: easy version}
(ZFC + PD)
Let $(\Kfr,\equiv, r)$ be a scattered projective ranked equivalence relation.
Then, relative to some oracle, for every $Z\in 2^\om$ and every $X\in \Kfr$ we have that if $r(X)<\om_1^Z$, then $Z$ computes a member of $[X]$.
\end{theorem}

This theorem alone is not enough to get that every scattered theory satisfies hyperarithmetic-is-recursive because, if a model $\A$ has high Scott rank, there will be a real $Z$ such that $\A$ is hyperarithmetic in $Z$ and with $SR(\A)\geq\om_1^Z$.

Theorem \ref{thm: easy version} is, however, strong enough to be applied in the following examples.
The first trivial example of a scattered projective ranked equivalence relation is the isomorphism relation on the class of countable ordinals, where the rank function is the identity.
Spector proved in \cite{Spe55} that every hyperarithmetic well-ordering has a computable copy.

A more interesting example is the relation of bi-embeddability on the class of countable linear orderings.
 (A linear ordering is {\em scattered} if it has no sub-ordering isomorphic to the rationals.)
 On the class of scattered linear orderings, we can use the Hausdorff rank as our ranking function.
 (The {\em Hausdorff rank} of a linear ordering $\L$ is the least $\a$ such that $\a$ embeds into a proper segment of $\ZZ^{\a+1}$. 
 This definition might differ slightly  from other  definitions in the literature.)
 The fact that there are only countably many equivalence classes of each Hausdorff rank follows from the work of Laver \cite{Lav71}.
 This is a scattered projective ranked equivalence relation.
 It is well known that the Hausdorff rank of a scattered linear ordering $\L$ is always less than $\om_1^\L$.
It follows from the theorem above that, for every oracle on a cone, every hyperarithmetic scattered linear ordering is bi-embeddable with a computable one.
Since every non-scattered linear ordering is bi-embeddable with $\QQ$, (which is computable), this is true for all linear orderings.
The author showed in \cite{MonEqui, MonBSL} that every hyperarithmetic linear ordering is bi-embeddable with a computable one relative to every oracle.
Getting the base of the cone to be $0$ still requires the techniques from \cite{MonEqui, MonBSL}.

A third example is the relation of bi-embeddability on the class of countable $p$-groups.
The ranking function here is the Ulm rank, and the fact that there are only countably many equivalence classes of each Ulm rank follows from the work on Barwise and Eklof \cite{BE70}.
If a $p$-group $\A$ has a sub-group isomorphic to $\ZZ(p^\infty)^\om$, then it is bi-embeddable with it.
Otherwise, it is proved in \cite[Section 7]{GMRanked} that the Ulm rank of $\A$ is less than $\om_1^\A$.
It follows from the theorem above that for every oracle on a cone, every hyperarithmetic $p$-group is bi-embeddable with a computable one.
Greenberg and Montalb\'an \cite[Theorem 1.11]{GMRanked} had shown that this holds for every oracle.
Again, the results of  Greenberg and Montalb\'an on  $p$-groups are still required to get the base of the cone to be $0$.

%
%

\subsection{Martin's Theorems} \label{se: martins}

Before proving Theorem \ref{thm: easy version}, we review the key lemmas of Martin that we will use.
We start by stating Turing determinacy.
A {\em pointed} tree is a perfect sub-tree $P$ of $2^{<\om}$, all whose paths compute $P$.
Given $X\in 2^\om$, we let $P(X)$ be the path of $P$ obtained by following $X$ at every split of $P$.
(That is, if we think of $P$ as the downward closure of the image of a function $p\colon 2^{<\om}\to 2^{<\om}$ which preserves the lexicographic ordering and inclusion, then $P(X)=\bigcup_np(X\upto n)$.)
Notice that when $P$ is pointed we have that for every $X\in 2^\om$, $P(X)\teq X\oplus P$.
A {\em cone} is a class of the form $\{X\in 2^\om: X\geqt Y\}$ for some $Y$.
If $P$ is pointed, then the closure of $[P]$ under Turing equivalence is exactly the cone above $P$.

\begin{lemma}[{Martin, see \cite{SS88}}] (ZFC+PD) \label{lemma: Martin pointed}
Let $\varphi\subseteq 2^\om$ be a projective class of reals cofinal in the Turing degrees (i.e.,  $\forall X\in 2^\om\ \exists Y\in 2^\om\ (\varphi(Y) \wedge Y\geqt X)$).
Then, there exists a pointed tree $P\subseteq 2^\om$ all whose paths satisfy $\varphi$.
\end{lemma}

\begin{corollary}[Turing-determinacy]  (ZFC+PD) \label{cor: TD}
Let $\varphi\subseteq 2^\om$ be a projective class of reals invariant under Turing equivalence.
Then either $\varphi$ contains a cone or $\varphi$ is disjoint from a cone.
\end{corollary}

\begin{lemma}[Martin](ZFC+PD) \label{lemma: Martin ordinals}
Let $f\colon 2^\om \to \om_1$ be a function invariant under Turing equivalence which has a projective presentation, i.e., there is a projective $g\colon 2^\om\to 2^\om$ such that $g(X)$ is a well-ordering isomorphic to $f(X)$.
If for every $X$, $f(X)< \om_1^X$, then $f$ is constant on a cone.
\end{lemma}
\begin{proof}
For each $X$, since $f(X)< \om_1^X$, there exists an $e\in\om$ such that $\{e\}^X$ is a well-ordering isomorphic to $f(X)$.
Therefore, there exists some $e$ such that on a cofinal set of $X$, we have $\{e\}^X \isom f(X)$.
Then, by Martin's lemma, there exists  a pointed tree $P$, such that for every $X\in[P]$, $g(X)\isom \{e\}^X$.
So, the map $Y\mapsto \{e\}^{P(Y)}$ is a continuous map from $2^\om$ to the set of well-orderings, which coincides with $f$ on the cone above $P$.
By $\Si^1_1$-boundedness, any such continuous map has to be bounded below $\om_1$, and hence $f$ is bounded in the cone above $P$.
Therefore, there has to be some ordinal $\a$ such that for an cofinal set of reals $X$, $f(X)=\a$.
By Turing determinacy again, $f(X)=\a$ on a cone.
\end{proof}

We are now ready to prove Theorem \ref{thm: easy version}.
Fix a scattered projective ranked equivalence relation $(\Kfr,{\equiv}, r)$.
If there are only countably many $\equiv$-equivalence relations, then $(\Kfr,\equiv)$ trivially satisfies hyperarithmetic-is-recursive on a cone  (just take an oracle that computes a member of each equivalence class).
Thus, let us assume that there are uncountably many equivalence classes, and hence that the image of $r$ is unbounded in $\om_1$.

\begin{proof}[Proof of Theorem \ref{thm: easy version}]
Suppose, toward a contradiction,  that on every cone there is a $Z$ for which there exists an $[X]\in \Kfr/\equiv$ with $r(X)<\om_1^Z$ but with no $Z$-computable members.
Then, by Turing determinacy (Corollary \ref{cor: TD}), there is a cone of such $Z$.
For each $Z$ in this cone, let $\a_Z$ be the least $\a< \om_1^Z$ for which there is an $X\in \Kfr$ such that $r(X)=\a$ and $Z$ does not compute any member of $[X]$.
By Lemma \ref{lemma: Martin ordinals}, the function $Z\mapsto \a_Z$ is constant on a cone.
This is not possible, because there are only countably many equivalence classes with rank less than a fixed $\alpha$, and any $Z$, sufficiently high up in the Turing degrees,  can compute members of all of them.
\end{proof}

\section{Minimal counterexamples} \label{se: minimal}

In this section we show that if $T$ is a minimal counterexample to Vaught's conjecture (see definition below), $T$ satisfies that hyperarithmetic-is-recursive on a cone.

\begin{definition}
We say that an $\L_{\om_1,\om}$ sentence $T$ is a {\em minimal counterexample to Vaught's conjecture} if it has $\aleph_1$ many models, it is scattered, and for any $\L_{\om_1,\om}$ sentence $\varphi$, either $T\wedge \varphi$ or $T\wedge\neg\varphi$ has countably many models.
\end{definition}

It is known that if there is a counterexample to Vaught's conjecture, then there is a minimal one (see  \cite[Theorem 1.5.11]{Ste78}).

\begin{theorem} \label{thm: minimal example}
Let $T$ be a minimal counterexample to Vaught's conjecture.
Then, there is an oracle $Z_0$ such that, for every $Z\geqt Z_0$, and every model $\A\models T$,  
\[
\om_1^{\A,Z_0}\leq \om_1^Z \ \ \iff \ \ Z\in\Sp(\A).
\] 
\end{theorem}

Observe that this says that, relative to $Z_0$, $\Sp(\A)=\{Z\in 2^\om:\om_1^Z\geq\om_1^\A\}$.
Note that this implies that $T$ satisfies hyperarithmetic-is-recursive on the cone above $Z_0$.
We also remark that the right-to-left direction is immediate.
The rest of this section is dedicated to prove the left-to-right direction.

\begin{lemma}\label{le: club}
Let $T$ be a minimal counterexample to Vaught's conjecture.
Then, there is a club of ordinals $C\subseteq \om_1$ such that for every $\a\in C$, any two models of $T$ of Scott rank at least $\a$ are $\Sii_{<\a}$-elementary equivalent.
\end{lemma}
This lemma is probably known.
We include a proof using some of the techniques from Section \ref{se: bf} for completeness. 
\begin{proof}
Let $C$ be the set of all $\a\in\om_1$ such that any two models of $T$ of Scott rank at least $\a$ are $\Sii_{<\a}$-elementary equivalent.

We show that $C$ is unbounded.
Let $\a_0< \om_1$ be any ordinal.
We will define an increasing sequence $\{\a_i:i\in\om\}$ whose limit will satisfy the desired property.
Suppose we have already defined $\a_i$.
For each $\equiv_{\a_i}$-equivalence class $\si$, there is a $\Pii_{\a_i+1}$ sentence $\bfpsi_\si$ defining the models in that equivalence class (see Section \ref{ss: extended language}).
Thus, since  $T$ is minimal, for each such equivalence class, either there are countably many models in that class, or countably many models outside the class. 
Therefore, all $\equiv_{\a_i}$-equivalence classes $\si$, except for one, have countably many models.
So, for those $\si$, there is a $\gamma_\si<\om_1$ such that every model of $T$ in the $\equiv_{\a_i}$-equivalence class $\si$ has Scott rank less than $\gamma_\si$.
Since $T$ is scattered, there are countably many such $\si$, and we can define $\a_{i+1}$ to be the least  ordinal that bounds all such $\gamma_\si$.
Let $\a=\sup_{i\in\om} \a_i$.

We claim that any two models of $T$ of Scott rank at least $\a$ are $\Sii_{<\a}$-elementary equivalent:
If $\b<\a$, then $\b<\a_i$ for some $i$, and hence any two models of $T$ of Scott rank at least $\a$ belong to the same $\equiv_{\a_i}$-equivalence class, and hence are $\Sii_\b$-elementary equivalent.

This last paragraph also serves to show that $C$ is closed under limits. 
\end{proof}

\begin{lemma}\label{lemma: Sacks 4.7}
Let $T$ be an  $\L^c_{\om_1,\om}$ counterexample to Vaught's conjecture.
Then, for every admissible $\a$, there is a model $\A\models T$ with $SR(\A)\geq \om_1^\A=\a$.
\end{lemma}
This is a well-known application  of  Gandy's basis theorem.
We include it for completeness.
\begin{proof}[Sketch of the proof]
Let $X$ be such that $\om_1^X=\a$.
For each $\b<\a$, there is an $\L^{c,X}_{\om_1,\om}$ sentence $S_\beta$ such that $\A\models S_\b$ if and only if $SR(\A)\geq \b$ (this sentence can be built using the back-and-forth relations from Section \ref{se: bf}).
Now, $\bar{T}=T\cup \{S_\b:\b<\om_1^X\}$ is a $\Pi^1_1(X)$ theory of $\L^{c,X}_{\om_1,\om}$, and has a model $\A$ since $T$ has models of arbitrary high Scott rank.
Being a model of a $\Pi^1_1(X)$ theory is a $\Si^1_1(X)$ property, and hence by Gandy's basis theorem (see \cite[III.1.5]{Sac90}), there is  a model $\A\models \bar{T}$ with $\om_1^\A\leq\om_1^X=\a$.
All models of $\bar{T}$ satisfy that $\a \leq SR(\A)$ and all  structures satisfy $SR(\A)\leq \om_1^\A+1$.
Since $\a$ and $\om_1^\A$ are limit ordinals, it follows that $\om_1^\A=\a$ and that $SR(\A)$ is either $\a$ or $\a+1$.
\end{proof}

\begin{corollary}\label{cor: Sacks 4.7}
Let $T$ be an  $\L^c_{\om_1,\om}$  counterexample to Vaught's conjecture.
Then, for every $Z$, there is a model $\A\models T$ with $SR(\A)\geq \om_1^{\A,Z}=\om_1^Z$.
\end{corollary}
\begin{proof}
Relativize the proof above to $Z$ and let $\a=\om_1^Z$.
\end{proof}

Recall from Section \ref{se: omega 1 CK} that if $SR(\A)\geq \om_1^{\A,Z}$, then $\om_1^{\A,Z}=\om_1^\A$.

\begin{lemma}\label{le: intermediate}
Let $Y,Z\in 2^\om$ be such that $\om_1^Y=\om_1^Z$.
There exists a $G\in 2^\om$ such that
\[
 \omega_1^Z = \omega_1^{Z\oplus G} = \omega_1^{G} = \omega_1^{G\oplus Y} = \omega_1^Y.
\]
\end{lemma}
\begin{proof}
Let $\H_Z$ be a $Z$-computable copy of $\om_1^Z\cdot(1+\QQ)$, let  $\H_Y$ be a $Y$-computable copy of $\om_1^Y\cdot(1+\QQ)$, both with domain $\om$, and let $f\colon\H_Z\to\H_Y$ be an isomorphism.
(The existence of $\H_Z$ and $\H_Y$ is due to Harrison \cite{Har68}.)
Notice that $f$ is a permutation of $\om$.
Let $g$ be a permutation of $\om$ which is  hyperarithmetically generic relative to $Y$, $Z$ and $f$.
Since $g$ is generic relative to $Z$, we have that $\om_1^{g\oplus Z}=\om_1^Z$ (this follows for instance from \cite[IV.3.6]{Sac90}).
Also, since $g$ is generic relative to $Y$ and $f$, $f\circ g$ is generic relative to $Y$, and hence $\om_1^{f\circ g\oplus Y}=\om_1^Y$.

Let $G$ be the pull-back of $\H_Z$ through $g$.
That is, $G$ is a copy of $\H_Z$, isomorphic to it via $g$.
Since $G$ has initial segments isomorphic to all the ordinals below $\om_1^Z$, we have that $\om_1^Z\leq \om_1^G$. 
Since $G\leqt g\oplus Z$, we have that $\om_1^G\leq \om_1^{G\oplus Z}\leq \om_1^{g\oplus Z}=\om_1^Z$.
Closing the circle we get 
\[
\om_1^G = \om_1^{G\oplus Z} = \om_1^Z.
\]
Notice that $G$ is also the pull-back of $\H_Y$ through $f\circ g$, and hence by the same argument $\om_1^G = \om_1^{G\oplus Y} = \om_1^Y$.
\end{proof}

\begin{proof}[Proof of Theorem \ref{thm: minimal example}]
Let $\KK$ be the class of models of $T$.
Then $(\KK, \isom, SR)$ is a projective scattered ranked equivalence relation.
So, by Theorem \ref{thm: easy version}, on a cone $\Cfr_4\subseteq 2^\om$ we have that every $Z$ computes a copy of all structures $\A\in\KK$ with $SR(\A)<\om_1^Z$.
Relativize the rest of the argument to the base of this cone.
Relativize again, if necessary, to make sure $T$ is a computable infinitary formula.

Let $\Cfr_3$ be the class of sets $Z$ which compute at least one $\A\in\KK$ with $SR(\A)\geq \om^{\A,Z}_1=\om_1^Z$.
This class is clearly cofinal in the Turing degrees:
given $Y$, there is an $\A$ with  $SR(\A)\geq \om^{\A,Y}_1=\om_1^Y$, and hence there is a $Z\geqt Y$, with $\om_1^Z=\om_1^Y$, which computes $\A$.
Thus, $\Cfr_3$ contains a cone; let $Z_3$ be the base of this cone.

Now, consider the class $\Cfr_0$ of sets $Z$ which compute a copy of all structures  $\A\in\KK$ with $SR(\A)\geq \om^{\A,Z_3}_1=\om_1^Z$.
We claim that this class is also cofinal:
Take $Z_2\in 2^\om$.
We can assume $Z_2\geqt Z_3$.
Let $\a$ be a $Z_2$-admissible ordinal that belongs to the club of  Lemma \ref{le: club}.
(Such an $\a$ exists because the class of $Z_2$-admissible ordinals contains a club, as for instance the set of all  $\a$ for which $L_\a[Z_2]$ is an elementary substructure of $L_{\om_1}[Z_2]$.)
Let $Z_1\geqt Z_2$ be such that 
\[
\om^{Z_1}_1=\a.
\]
We claim that $Z_1\in \Cfr_0$.
Let $\A\models T$ such that $SR(\A)\geq \om^{\A,Z_3}_1=\om_1^{Z_1}$ be given by Corollary \ref{cor: Sacks 4.7}.
Since $Z_1\geqt Z_3$, $Z_1$ computes at least one  structure $\C$ with $SR(\C)\geq \om_1^{\C,Z_1}= \om_1^{Z_1}$.
(Notice that $\om_1^{\C,Z_1}=\om_1^{\C,Z_3}=\om_1^\C$, because $SR(\C)\leq\om_1^\C+1$.)
We claim that $\C\isom \A$.
Let $Y\geqt Z_3$ be such that $\om_1^Y=\om_1^{Z_1}$ and $Y$ computes $\A$.
Let $G\geqt Z_3$ be such that 
\[ 
\a =\om_1^{Z_1}= \omega_1^{Z_1\oplus G} = \omega_1^{G} = \omega_1^{G\oplus Y} = \omega_1^Y.
\]
The existence of such $G$ is given by Lemma \ref{le: intermediate} relativized to $Z_3$.

Since $G\geqt Z_3$, $G$ computes at least one structure $\B\in \KK$ with $SR(\B)\geq \om^{\B}_1=\om_1^{G}$.

So, we have that $Y$ computes $\A$, $G$ computes $\B$ and $Z_1$ computes $\C$.
All these structures are models of $T$ and have Scott rank at least $\a$.
By the choice of $\a$, this implies that these three structures are $\Sii_{<\a}$-elementary equivalent.

Since $\om_1^{Z_1\oplus G}=\a$, $\L^{c,Z_1\oplus G}_{\om_1,\om}\subseteq  \Sii_{<\a}$.
Since ${Z_1\oplus G}$ computes both $\B$ and $\C$, it follows from Theorem \ref{thm: effective Scott}  that $\C$ and $\B$ are isomorphic.
Analogously, $\om_1^{G\oplus Y}=\a$ and $G\oplus Y$ computes both $\A$ and $\B$ so they are isomorphic.
Thus $\A\isom \C$.

This proves that $\Cfr_0$ is unbound in the Turing degrees, and hence it contains a cone with base $Z_0$.
Now, take $Z$ in this cone and $\A$ with $\om_1^{\A,Z_0}\leq \om_1^Z$.
If $SR(\A)<\om_1^Z$, then $Z$ computes a copy of $\A$ because $Z\in \Cfr_4$.
If $SR(\A)\geq \om_1^Z$, then 
\[
\om_1^Z+1 \geq \om_1^{\A,Z_0}+1 \geq \om_1^{\A,Z_3}+1\geq SR(\A) \geq \om_1^Z,
\]
and hence $SR(\A)\geq \om_1^{\A,Z_3}= \om_1^Z$.
It then follows from the definition of $\Cfr_0$ that $Z$ computes a copy of $\A$.
\end{proof}

Relativizing the the oracle $Z_0$ given by the previous proof, we remark that for every admissible ordinal $\a$, by Lemma \ref{lemma: Sacks 4.7}, there is a structure $\A\models T$ with $\om_1^\A=\a$, and hence with spectrum $\{Z\in 2^\om: \om_1^Z\geq \a\}$.
It follows that every minimal counterexample to Vaught's conjecture satisfies part (\ref{main: part 3}) in Theorem \ref{thm: main}.

\section{The fine back-and-forth structure} \label{se: bf}

In this section we describe the back-and-forth structure of a class of structures $\KK$.
We will show how, from this back-and-forth structure, we can computably build members of $\KK$ with prescribed back-and-froth types.
This will give us the key lemma used in the proof of Theorem \ref{thm: counter example}.

There are two non-equivalent definitions of the back-and-forth equivalence relations in the literature.
We use the finer version, which was introduced by Karp and is used in \cite[Chapter 15]{AK00}.
The notion in \cite{Bar73} is too coarse for our purposes.

\subsection{Back-and-forth relations}

The back-and-forth relations measure how hard it is to differentiate two structures or two tuples from the same structure.
The idea is that two tuples are $\a$-back-and-forth equivalent if we cannot differentiate them using only $\a$ Turing jumps.

Before giving the formal definition, we need a bit of notation.
If $\L$ is a language with infinitely many symbols, let $\L\upto k$ denote the first $k$ symbols in $\L$.
Recall that, without loss of generality, we assume $\L$ is a relational language.
If $\abar$ is a tuple of elements of $\A$, we abuse notation and write $\abar\in \A$.
We let $D_\A(\abar)$ be the $\L\upto k$-atomic diagram of $\abar$ in $\A$, where $k$ is the length of $\abar$.
Notice that there are only finitely many atomic formulas which  use only the first $k$ symbols in $\L$ and $k$ variables.

\begin{definition}\label{def: bnf}
We now define the {\em $\a$-back-and-forth relations} on tuples of $\L$-structures by induction on $\a$.
Let $\A,\B$ be two $\L$-structures, and let $\abar\in\A$, $\bbar\in\B$ be tuples of the same length $k$.
We say that $(\A,\abar)\leq_0(\B,\bbar)$ if $\abar$ and $\bbar$ satisfy the same $\L\upto k$-atomic formulas (i.e., if $D_\A(\abar)=D_\B(\bbar)$).
We say that $(\A,\abar)\leq_{\a}(\B,\bbar)$ if for every $\dbar\in\B$ and every $\g<\a$, there exists $\cbar\in\A$ such that $(\A,\abar\cbar)\geq_\g(\B,\bbar\dbar)$.
\end{definition}

The following theorem states three equivalent definitions of these relations showing their naturality.
For a tuple $\abar\in\A$, {\em the $\Pii_\a$-type of $\abar$ in $\A$} (denoted by $\Pii_\a\tp_\A(\abar)$) is the set of all infinitary $\Pii_\a$ formulas true of $\abar$ in $\A$.

\begin{theorem}[{Karp; Ash and Knight \cite[15.1, 18.6]{AK00}}]  \label{thm: def bnf}
For $\a\geq 1$, the following are equivalent.
\begin{enumerate}
\item $(\A,\abar)\leq_{\a}(\B,\bbar)$,
\item $\Pii_\a\tp_\A(\abar) \subseteq \Pii_\a\tp_\B(\bbar)$,
\item If we are given a structure $(\C,\cbar)$ that we know is isomorphic to either $(\A,\abar)$ or  $(\B,\bbar)$, deciding whether it is isomorphic to $(\A,\abar)$ is (boldface) ${\mathbf \Si}^0_\a$-hard.
That is, for every ${\mathbf \Si}^0_\a$ subset $\Cfr\subseteq 2^\om$, there is a continuous operator $F\colon 2^\om\to\KK$ such that, $F(X)$ produces a copy of $(\A,\abar)$ if $X\in \Cfr$, and a copy of $(\B,\bbar)$ otherwise.   \label{part: def bnf 3}
\end{enumerate}
\end{theorem}

(Statement (\ref {part: def bnf 3}) is not exactly \cite[Theorem 18.6]{AK00}, but it can be derived from it by relativizing; see \cite{HMapprox}.)

The relation $\leq_\a$ is a pre-ordering on $\{(\A,\abar): \A\in\KK, \abar\in\A\}$, and it induces an equivalence relation and a partial ordering on the quotient as usual:
We let $(\A,\abar)\equiv_{\a}(\B,\bbar)$ if $(\A,\abar)\leq_{\a}(\B,\bbar)$ and $(\A,\abar)\geq_{\a}(\B,\bbar)$.
We define $\BFI{\a,k}(\KK)$ to be the quotient partial ordering:
\[
\BFI{\a,k}(\KK)  =  \frac{\{(\A,\abar): \A\in\KK, \abar\in\A^k\}}{\equiv_\a},
\]
which is  ordered by $\leq_\a$ in the obvious way.
We will write just $\BFI{\a,k}$ when $\KK$ is understood.
We let $\BFI\a=\bigcup_{k\in\om}\BFI{\a,k}$.
If $\si$ is the $\equiv_\a$-equivalence class of $(\A,\abar)$, we say that $\abar$ has {\em $\a$-bf-type $\si$}.

Notice that when $T$ is scattered and $\KK$ is the class of models of $T$, the sets $\BFI{\a,k}(\KK)$ are all countable.
In \cite{McountingBF} the author analyzed some computability theoretic properties of $\KK$ that depend on whether $\BFI\a(\KK)$ is countable or not.
One of the ideas, that started in \cite{McountingBF} and that we wish to impart in this paper too, is that the partial ordering $(\BFI\a(\KK),\leq_\a)$ can give useful information about $\KK$.

We define some operations on $\BFI\a$.
These are all very basic operations.

First, we observe that if $\b<\a$, then $(\A,\abar)\equiv_{\a}(\B,\bbar)$ implies $(\A,\abar)\equiv_{\b}(\B,\bbar)$.
Thus, there is a natural projection from $\BFI\a$ to $\BFI\b$.
Given $\si\in\BFI\a$, we let $(\si)_\b\in\BFI\b$ be the image of this projection, i.e., $(\si)_\b$ satisfies that for every $(\A,\abar)$ of $\a$-bf-type $\sigma$, $(\A,\abar)$ has $\b$-bf-type $(\si)_\b$.

Secondly, we consider the operation of re-arranging a tuple.
For $k\leq \ell$, given $\abar=(a_1,...,a_\ell)\in A^\ell$ and  $\iota = (\iota_1,...,\iota_k)\in \{1,...,\ell\}^k$, let $\pi_\iota(\abar)=(a_{\iota_1},...,a_{\iota_k})\in A^k$.
Of course, $(\A,\abar)\equiv_{\a}(\A,\bbar)$ implies $(\A,\pi_\iota(\abar))\equiv_{\a}(\B,\pi_\iota(\bbar))$.
Thus, $\iota$ induces a natural projection from $\BFI{\a,\ell}$ to $\BFI{\a,k}$.
Given $\si\in\BFI{\a,\ell}$ and $\iota \in \{1,...,\ell\}^k$, we let $\pi_\iota(\si)\in\BFI{\a,k}$ be the image of this projection, i.e., $\pi_\iota(\si)$ satisfies  that for every $(\A,\abar)$ of $\a$-bf-type $\sigma$, $(\A,\pi_\iota(\abar))$ has $\a$-bf-type $\pi_\iota(\si)$.

The simplest case for the operation above is when we want to consider an initial segment of a tuple.
If $\iota=(1,...,k)\in \{1,...,\ell\}^k$, then $\pi_\iota(a_1,...,a_\ell)=(a_1,...,a_k)$.
In this case, if  $\pi_\iota(\si)=\tau$, we write $\tau\subseteq \si$.

We use $|\cdot|$ to denote the size of a tuple.
That is, if $\si$ is the $\b$-bf-type of a tuple $\abar=(a_1,...,a_k)$, we let $|\si|=|\abar|= k$.

Last, we define a less trivial operation that is key to understanding the structure of the bf-types.
For $\b$ and $\abar\in\A\in\KK$, we define the set $\ext_\b(\A,\abar)\subseteq \BFI{\b}(\KK)$ to be the $\leq_{\b}$-downward closure in $\BFI{\b}(\KK)$ of the set of $\b$-bf-types of tuples of the form $(\A,\abar\dbar)$. 
The important feature of this definition is that  
\[
 (\A,\abar)\leq_\g(\B,\bbar) \quad \iff \quad (\forall\b<\g)\ \ext_\b(\A,\abar)\supseteq\ext_\b(\B,\bbar),
\]
which follows immediately from the definitions of $\leq_\g$ and $\ext_\b$.

Now, given $\sigma\in\BFI\g(\KK)$ and $\b<\g$, we let $\ext_\b(\sigma)=\ext_\b(\A,\abar)$ for any$(\A,\abar)$ of $\a$-bf-type $\sigma$.
It follows that for $\sigma,\tau\in\BFI\a(\KK)$, 
\[
\sigma\leq_\g \tau \quad \iff \quad (\forall\b<\g)\ \ext_\b(\si)\supseteq\ext_\b(\tau).
\]

This definition is off by one from the definition in \cite{McountingBF}--what there is called $\ext_\a$, here is called $\ext_{\a-1}$.
The old definition does not work well when $\a$ is a limit ordinal, a case not considered in \cite{McountingBF}.


\begin{definition} \label{def: b&f structure}
Given an ordinal $\a$, we refer to the  family of structures 
\[
\left\{ (\BFI {\b,k}(\KK); \leq_\b)\ : \ \ \b\leq \a, k\in\om\right\},
\]
together with 
\begin{itemize}
\item the projections $(\cdot)_{\b}\colon \BFI{\geq \b}\to\BFI{\b}$ for all $\b<\a$;
\item the projections $\pi_\iota(\cdot)\colon \BFI{\b,k}\to\BFI{\b,\ell}$ for all $\b\leq \a, k,\ell\in\om, \iota\in \{1,...,\ell\}^k$; 
\item the binary relations $\{(\tau,\si): \tau\in \ext_\b(\si)\} \subseteq \BFI{\b}\times\BFI{>\b}$ for all $\b<\a$; and
\item the map $D(\cdot)$ that assigns to each 0-bf-type $\sigma$, its $\L\upto {|\sigma|}$-atomic diagram, (that is, $D_\A(\abar)$ for any $(\A,\abar)$ with $0$-bf-type $\si$,)
\end{itemize}
as the {\em $\a$-back-and-forth structure of $\KK$}.

The $\a$-back-and-forth structure of $\KK$ can be viewed as a single 1st-order structure (over a language that is computable in any given presentation of the ordinal $\a$).
We say that $\KK$ has a {\em computable $\a$-back-and-forth structure} if this structure has a computable presentation.

We will sometimes refer to the $\a$-back-and-forth structure as the {\em $\a$-bf-structure}.
We define the {\em $(<\a)$-bf-structure} of a class in the obvious ways considering $(\BFI {\b,k}(\KK); \leq_\b)$ only for $\b<\a$.
\end{definition}

For the case when $\a<\om$, this definition is essentially the same as the one given in \cite[2.3]{McountingBF}, except for some minor cosmetic changes.

We do know some examples of nice classes of structures where the $\a$-bf-structure is computable.
In \cite{McountingBF} we show that for the class of linear orderings, the 2-bf-structure is computable and for the class of equivalence structures, the 1-bf-structure is computable.
Harris and the author \cite{HM} show that the $(<\om)$-bf-structure of Boolean algebras is computable.

\subsection{The extended language} \label{ss: extended language}

\begin{definition}
Let $\Bb$ be a presentation of the  $\a$-bf-structure of $T$.
Let 
\[
\L^\Bb_\a=\L\cup\{\bfphi_\si: \si\in \BFI{\leq \a}\},
\]
where each $\bfphi_\si$ is a $|\si|$-ary relation.
Note that if $\BB$  is $X$-computable, $\L^\Bb_\a$ is an $X$-computable language.
\end{definition}

If $\A$ is a model of $T$, we let $\A_{(\a)}$ be the extension of $\A$ to an $\L^\Bb_\a$-structure defined as follows:
For each $\si\in\BFI{\b}$,
\[
\A_{(\a)}\models\bfphi_\si(\xbar)\iff \si\leq_\b(\A,\xbar),
\]
(where $\si\leq_\b(\A,\xbar)$ means that $(\B,\bbar)\leq_\b(\A,\abar)$ for all $(\B,\bbar)$ which have $\b$-bf-type $\si$).

We showed in \cite{McountingBF} that if $\Bb$ is $X$-computable, and $\si\in\BFI{\b}$, then $\bfphi_\si$ is equivalent to a $\Pico{X}_\b$-computable formula.
The reason is that for $\b=0$, $\bfphi_\si(\xbar) \iff$ ``$D(\xbar)=D(\si)$,'' and for $\b>0$, 
\begin{equation}
\bfphi_\si(\xbar)\ \ \iff\ \ 
	\bigwedge_{\g<\b}\bigwedge_{\substack{\tau\in \BFI{\g}\\ \tau\not\in\ext_\g(\si)}} \forall\ybar \neg \bfphi_\tau(\xbar,\ybar).
\end{equation}
It follows that if $X$ computes a copy of $\A$, then $\A_{(\a)}$ has a copy which is $\Delta^0_{\a+1}(X)$.

We also showed in \cite{McountingBF} that the formulas $\{\bfphi_\si: \si\in \BFI{\leq \a}\}$ form a complete set of $\Pii_\a$-formulas for $\KK$, in the sense that every $\Sii_{\a+1}$-$\L$-formula is equivalent (in all structures in $\KK$) to a $\Sii_1$-$\L_\a^\BB$-formula.

For $\b<\a$ and $\tau\in \BFI{\b}$, we define $\bfpsi_\tau(\xbar)$ by 
\[
\A\models\bfpsi_\tau(\xbar)\iff \tau\equiv_\b(\A,\xbar).
\]
Since $(\A,\abar)\leq_\a (\C,\cbar)$ implies $(\A,\abar)\equiv_\b (\C,\cbar)$ for any $\b<\a$, we can define $\bfpsi_\tau$ in $\L_\a$ by
\begin{equation}
\bfpsi_\tau(\xbar) \ := \ \bigvee_{\substack{\si\in \BFI{\a}, \\ (\si)_\b=\tau}} \bfphi_\si(\xbar).
\end{equation}
On the other hand, using that $(\A,\abar)\leq_\b \si$ if and only if, for all $\g<\b$, $\ext_\g(\si)\subseteq \ext_\g(\A,\abar)$, we get that
\begin{equation}   \label{eq: def psi}
\bfpsi_\si(\xbar)  \ \ \iff \ \  \bfphi_\si(\xbar) \wedge \bigwedge_{\g<\b}\bigwedge_{\tau\in\ext_\g(\si)} \exists \ybar \bfphi_\tau(\xbar,\ybar).
\end{equation}

Now, we define a theory that captures all this.

\begin{definition}
Let $T^\Bb_\a$ be the following family of sentences:
\begin{enumerate}\renewcommand{\theenumi}{T{\arabic{enumi}}}
\item (the existence and uniqueness of bf-types for all tuples) \label{axioms: part 1}
	\begin{itemize}  
	\item $\forall \xbar \bigvee_{\si\in \BFI{\a}}  \bfphi_\si(\xbar)$,
	\item $\bigwedge_{\b<\a}\bigwedge_{k\in\om}\bigwedge_{\tau,\tau'\in\BFI{\b,k}, \tau\neq\tau'}\forall \xbar\ \neg\left(\bfpsi_\tau(\xbar)\wedge\bfpsi_{\tau'}(\xbar) \right)$.
	\end{itemize}
\item (the recursive definition of $\bfphi_\si$)\label{axioms: part 2}
	\begin{itemize}
	\item $\bigwedge_{\si\in \BFI{0}} \ (\bfphi_\si(\xbar) \iff D(\xbar)=D(\si))$,
	\item $\bigwedge_{\b\leq\a}\bigwedge_{\si\in \BFI{\b}} \forall \xbar\ \left(\bfphi_\si(\xbar)\iff \bigwedge_{\g<\b}\bigwedge_{\tau\in \BFI{\g}\setminus\ext_\g(\si)} \forall\ybar \neg \bfphi_\tau(\xbar,\ybar)\right)$,
	\end{itemize}
\item (and the implications of $\bfpsi_\si$)\label{axioms: part 3}
	\begin{itemize}
	\item $\bigwedge_{\b<\a}\bigwedge_{\si\in \BFI{\b}} \forall \xbar\ \left(\bfpsi_\si(\xbar) \implies \bfphi_\si(\xbar) \wedge \bigwedge_{\g<\b}\bigwedge_{\tau\in\ext_\g(\si)} \exists \ybar \bfphi_\tau(\xbar,\ybar)\right)$.
	\end{itemize}
\end{enumerate}

We will write $\L_\a$ and $T_\a$ instead of $\L_\a^\Bb$ and $T_\a^\Bb$ when $\Bb$ is understood from the context.
We write $\L_{<\a}$ for $\bigcup_{\b<\a}\L_\b$.
\end{definition}

First, we observe that $T_\a$ is a  $\Pii_2$ theory.
Furthermore, it is $\Pic_2$ relative to the given presentation of $\Bb$.
Notice that for every model  $\A\models T$ we have that $\A_{(\a)}\models T_\a$.
However, if $\a$ is not large enough, $T_\a$ does not necessarily imply  $T$.
We will show that if $T$ is a $\Pii_\b$ theory for some $\b\leq \a$, then $T_\a$ does imply $T$.
So, for such $\a$ we have that the models of $T_\a$ are in a one-to-one correspondence with the models of $T$.
In particular, if $T$ is a counterexample to Vaught's conjecture, so is $T_\a$,  with the added property that $T_\a$ is $\Pii_2$.

\begin{lemma}\label{le: T alpha}
Fix an ordinal $\a$.
Let  $\si\in \BFI{\b}$ for some $\b\leq\a$, let $\B\models T_\a$, and let $\bbar\in \B$. 
\begin{itemize}
\item If $\B\models \bfphi_\si(\bbar)$,  then $\si\leq_\b(\B,\bbar)$. 
\item If $\b<\a$ and $\B\models \bfpsi_\si(\bbar)$,  then $\si\equiv_\b(\B,\bbar)$.
\end{itemize}
\end{lemma}
\begin{proof}
The proof is by simultaneous transfinite induction on $\b$.
The lemma is clearly true for the case for $\b=0$ because $\si$ determines $D_\B(\bbar)$.

Suppose first that $\B\models \bfphi_\si(\bbar)$.
Let $(\A,\abar)$ have $\b$-bf-type $\si$.
Consider $\g<\b$ and let $\dbar\in \B$; we need to find $\cbar\in\A$ such that $(\A,\abar,\cbar)\geq_\g(\B,\bbar,\dbar)$.
Axiom (\ref{axioms: part 1}) implies that there exists a $\tau\in\BFI{\g}$ such that $\B\models\bfpsi_\tau(\bbar,\dbar)$.
By the inductive hypothesis we have that $(\B,\bbar,\dbar)\equiv_\g \tau$.
Using (\ref{axioms: part 3}), we get that also $\B\models\bfphi_\tau(\bbar,\dbar)$, and since $(\B,\bbar)\models \bfphi_\si(\bbar)$, using the axiom (\ref{axioms: part 2}), we have that $\tau\in\ext_\g(\si)$.
Since $(\A,\abar)\equiv_\b\si$, we have that there exists $\cbar$ such that $(\A,\abar,\cbar)\geq_\g\tau$, and hence  $(\A,\abar,\cbar)\geq_\g  (\B,\bbar,\dbar)$ as wanted.

Suppose now that $\B\models\bfpsi_\si(\bbar)$.
From (\ref{axioms: part 3}) we get that also $\B\models\bfphi_\si(\bbar)$, and hence that $\si\leq_\b(\B,\bbar)$.
We now need to show the other inequality; that is, that for every $\cbar\in \A$ and every $\g<\b$, there exist $\dbar\in \B$ such that $(\B,\bbar,\dbar)\geq_\g(\A,\abar,\cbar)$.
Since $(\A,\abar)\equiv_\b\si$, we have that there exists $\tau\in\ext_\g(\si)$ such that $(\A,\abar,\cbar)\equiv_\g\tau$.
Since  $\B\models\bfpsi_\si(\bbar)$, by (\ref{axioms: part 3}), there is a $\dbar$ such that $\B\models\bfphi_\tau(\bbar,\dbar)$.
Then, by the inductive hypothesis, $(\B,\bbar,\dbar)\geq_\g\tau \equiv_\g (\A,\abar,\cbar)$ as wanted.
\end{proof}

\begin{corollary}
If $T$ is a $\Pii_{\a}$ theory and $\C\models T_\a$, then $\C$ models $T$.
\end{corollary}
\begin{proof}
By (\ref{axioms: part 1}), there is some $\si\in\BFI{\a,0}$ such that $\C\models \bfphi_\si$.
Let $\A\models T$ have $\a$-bf-type $\si$.
By the previous lemma $\A\leq_\a\C$, so $\C$ satisfies all the $\Pii_\a$ sentences true in $\A$.
It follows that $\C\models T$.
\end{proof}

\subsection{Building models out of back-and-forth structures}

In this section we prove a key lemma for our proofs in Section \ref{se: main}, namely that we can use the bf-structure to computably build structures with a prescribed bf-type (Lemma \ref{lemma: building models}).

Given a class of structures $\KK$, let $\KK^{fin} = \{D_\A(\abar): \abar\in \A, \A\in\KK\}$.
That is, $\KK^{fin}$ is the set of all finite atomic diagrams over finite restrictions of the language which  occurr in $\KK$.

\begin{lemma}\cite{McountingBF}  \label{le: builing pi 2}
If $T$ is $\Pic_2$, $\KK$ is the set of all models of $T$, and $\KK^{fin}$ is c.e., then $T$ has a computable model.
\end{lemma}
\begin{proof}
A full proof is given in \cite[Lemma 2.9]{McountingBF}; we sketch a proof here for completeness.

We will build a sequence of finite atomic diagrams $D_0\subseteq D_1\subseteq\cdots$, all in $\KK^{fin}$, and then define $\A$ by letting its atomic diagram be $D(\A)=\bigcup_s D_s$.
We will think of each $D_s$ as a finite $\L\upto |\A_s|$-structure $\A_s$ with $D(\A_s)=D_s$, and of $\A$ as $\bigcup_s\A_s$.

Write $T$ as $\bigwedge_{i\in\om}\forall\ybar_i \varphi_i(\ybar_i)$, where each $\varphi_i$ is $\Sic_1$.
At stage $s$, consider the least pair $(i,\abar)$, where $\abar$ is a tuple in $\A_s$ of size $|\ybar_i|$, that has not been considered so far. 
Search for a finite structure $\A_{s+1}\in \KK^{fin}$ extending $\A_s$ such that $\A_{s+1}\models \varphi_i(\abar)$.
Such a structure must exist because if $\B$ is a model of $T$ which has $\A_s$ as a substructure, then $\B\models \varphi_i(\abar)$, and hence some finite substructure of $\B$ does too.

In the limit, for every $i$ and every $\abar\in \A$, $\A\models\varphi_i(\abar)$, and hence $\A\models T$.
\end{proof}

\begin{lemma}\cite{McountingBF}\label{lemma: building models}
Let $\Bb$ be a presentation of the  $\a$-bf-structure of $T$ and $\si\in \BFI{\a,k}$.
Suppose that $T$ is $\Pii_\a$-axiomatizable.
There is a structure $(\A,\abar)$ of $\a$-bf-type $\si$ which is computable in $\Bb$.
\end{lemma}
\begin{proof}
Assume $\BB$ is computable.
Add to $\L_\a$ a finite tuple of constants $\cbar$ of size $|\si|$.
Now, consider
\[
T_{\a,\si} = T^\Bb_\a\  \cup\ \{\bfphi_\si(\cbar)\}\ \cup\ \{\bigwedge_{\g<\a}\bigwedge_{\tau\in\ext_\g(\si)} \exists \xbar \bfphi_\tau(\cbar,\xbar) \}.
\]
So, we have that if $(\A_{(\a)},\abar)\models T_{\a,\si}$, then $\A_{(\a)}$ models $T^\Bb_\a$, and hence $T$.
Furthermore, from equation (\ref{eq: def psi}) and the proof of Lemma \ref{le: T alpha} we get that $\cbar$ has $\a$-bf-type $\si$.

We now want a computable model of $T_{\a,\si}$.
Let $\KK$ be the class of all models of $T_{\a,\si}$.
We claim that $\KK^{fin}$ is c.e.
Notice that the lemma would then follow from the previous lemma.

For each $\tau\in\BFI{\a}$, let $D_\a(\tau)$ be the $\L_\a$-diagram of any tuple $(\A,\abar)$ of $\a$-bf-type $\tau$.
Notice that $D_\a(\tau)$ can be computed from $\tau$ and $\BB$ as follows:
the $\L$-part of the diagram is determined by $D((\tau)_0)$ and $\A_{(\a)}\models \bfphi_\d(\cbar,\abar)$ for $\d\in\BFI{\b}$ if and only if $\d\leq_\b (\tau)_\b$.

If $\bbar$ is a tuple in a model of $\A_{(\a)}$, then $\cbar\bbar$ has $\a$-bf-type $\tau$ for some $\tau\supseteq\si$, $\tau\in \BFI{\a}$.
Conversely, for every $\tau\supseteq\si$, there is a model $\A_{(\a)}\models T_{\a,\si}$ with a tuple of $\a$-bf-type $\tau$.
It follows that $\KK^{\fin}$ is determined by $\{ D_\a(\tau): \tau\in \BFI{\a}, \tau\supseteq \si\}$.
\end{proof}

We now prove that being the $\a$-bf-structure of $T$ is a $\Pi^1_1$ property; another lemma that will be useful later.

\begin{lemma}
The set of triples $(T,\a,\Bb)$ such that $\Bb$ is the $\a$-bf-structure of $T$ is a $\Pi^1_1$ set.
\end{lemma}
\begin{proof}
Let us start proving that given an ordinal $\a$, a countable list $\KK$  of $\L$-structures, and $\Bb$, deciding whether $\Bb$ is  the $\a$-bf-structure of $\KK$ is $\Delta^1_1$.
Let $\Cc$ be the $\a$-bf-structure of $\KK$.
The construction of $\Cc$ is clearly $\Delta^1_1$; we now need to check if $\BB$ is isomorphic to $\Cc$.
We do it level by level.
For $\b=0$, each $0$-bf-type corresponds to a finite diagram $D(\abar)$ occurring in $\KK$, so it is not hard to verify that $\BB$ considers them all, and to build an isomorphism between $\BFI{0}^\BB$ and $\BFI{0}^\Cc$.
Suppose now that we have an isomorphism between $\BFI{<\b}^\BB$ and $\BFI{<\b}^\Cc$, and we want to extend it to level $\b$.
For each $\si\in \BFI{\b}^\BB$, look for $\hat{\si}\in \BFI{\b}^\Cc$ such that for each $\g<\b$, $\ext_\g(\si)^\BB$ and $\ext_\g(\hat{\si})^\Cc$ are equal under the isomorphism that we have so far.
If no such $\hat{\si}$ exists, then $\BB$ is not correct.
Otherwise, this allows us to define a map from $\BFI{\b}^\BB$ to $\BFI{\b}^\Cc$, which we need to check is an isomorphism.
This procedure is clearly $\Delta^1_1$.

Now, suppose that we have a theory $T$, instead of a countable list $\KK$.
Using the previous lemma build, out of $\BB$, a list $\KK$ of models of $T^\Bb_\a$, realizing each $\a$-bf-type $\si$ in $\BFI{\a}^\Bb$.
Then check that all the models in $\KK$ satisfy $T$, and that $\Bb$ is actually the $\a$-bf-structure of $\KK$.
So far, this is still $\Delta^1_1$.
Finally, verify that every model of $T$ is $\equiv_\a$-equivalent to a model in $\KK$.
This last part is $\Pi^1_1$.
\end{proof}

\section{The main theorem} \label{se: main}


We now show that (\ref{main: part 1}) implies (\ref{main: part 3}) in Theorem \ref{thm: main}.
We will prove that  (\ref{main: part 2}) implies (\ref{main: part 1}) in Section \ref{se: coding}.

\begin{theorem} \label{thm: counter example}
Let $T$ be an $\L_{\om_1,\om}$ counterexample to Vaught's conjecture.
Then, there is an oracle $Z_0$ such that, for every $Z\geqt Z_0$ and every $\A\models T$,   
\[
\om_1^{\A,Z_0}\leq \om_1^Z \ \ \iff \ \ Z\in \Sp(\A).
\] 
\end{theorem}

\begin{proof}
The left-to-right direction is immediate.
We prove the other direction. 

Informally, the main steps behind this proof are as follows.
First we use Theorem \ref{thm: easy version} to find an oracle relative to which we can compute the $\b$-bf-structure of $T$ whenever we can compute $\b$.
We then use an overspill argument to show that if $Z$ can compute all the $\b$-bf-structures of $T$ for $\b<\om_1^Z$, then it can compute the $\a^*$-bf-structure over some non-standard ordinal $\a^*$ with well-founded part $\om_1^Z$.
Then, given an $\a^*$-bf-type $\si^*$, we use Lemma \ref{lemma: building models} to computably build a model of $T$ with $\a^*$-bf-type $\si^*$.
Even if we cannot make sense of what having $\a^*$-bf-type $\si^*$ means, we can show using Theorem \ref{thm: effective Scott} that any two $Z$-computable structures having the same $\a^*$-bf-type must be isomorphic.
This last bit requires a trick using non-standard models of $ZFC$ to put ourselves in the hypothesis of Theorem \ref{thm: effective Scott}.

First we prove that, on a cone, for every $X$ and every $\b<\om_1^X$, $X$ computes a presentation of the $\b$-bf-structure of $T$.
Let $\Kfr$ be the union over all $\b\in\om_1$ of the class of all presentations $\Bb$ of the $\b$-bf-structure  of $T$.
Given  such a $\b$-bf-structure $\Bb$, let $r(\Bb)=\b$.
Then $(\Kfr,\isom,r)$ is a projective scattered ranked equivalence relation.
So, by Theorem \ref{thm: easy version}, there is a cone as wanted.
We relativize the rest of the proof to this cone.

Now, consider $Z\in 2^\om$ and $\A\models T$  such that $\om_1^\A\leq\om_1^Z$.
We want to show that $Z$ computes a copy of $\A$.
Let $\a=\om_1^Z$.
Let $Y$ be such that $Y$ computes a copy of $\A$ and  $\om_1^Y=\om_1^Z$.
Let $G$ be  as in Lemma \ref{le: intermediate}, i.e., $G$ satisfies that 
\[
\a= \omega_1^Z = \omega_1^{Z\oplus G} = \omega_1^{G} = \omega_1^{G\oplus Y} = \omega_1^Y.
\]

We start by proving that $G$ computes a copy of $\A$, and then a symmetrical argument will show that $Z$ computes a copy of $\A$ too.

Let $\M_{G,Y}$ be an $\omega$-model of (enough of) $ZFC$ which contains $Y$, $G$ and all the ordinals below $\a$, but omits $\a$, (the existence of such a model is due to H. Friedman (see \cite[III,1.9]{Sac90}) and follows from an application of Gandy's theorem).
For every $\b<\a$, $G$ computes a copy $\Bb_\b$ of the $\b$-bf-structure of $T$.
These $\Bb_\b$ are in $\M_{G,Y}$, and are still the $\b$-bf-structure of $T$ within $\M_{G,Y}$ because this is a $\Pi^1_1$ property and all the reals of $\M_{G,Y}$ are reals in the real universe too.
We also know that any two $\b$-bf-structures of $T$ are isomorphic, and this is true outside and inside $\M_{G,Y}$ (because it is provable in $ZFC$).
Now, since $\M_{G,Y}$ omits $\a$, there is a non-well-ordered $\a^* \in On^{\M_{G,Y}}$, such that $G$ computes a structure $\Bb_{\a^*}$, which $\M_{G,Y}$ believes is an $\a^*$-bf-structure for $T$, as otherwise $\a$ would be definable in $\M_{G,Y}$.
By the uniqueness of $\b$-bf-structures, for $\b<\a$, the initial segment of $\Bb_{\a^*}$ of length $\b$ is the correct $\b$-bf-structure of $T$.
Since $\A$ is computable from $Y$, it belongs to $\M_{G,Y}$.
In $\M_{G,Y}$, this model has some $\a^*$-bf-type $\si^*\in\BFI{\a^*}$, and hence, by Lemma \ref{lemma: building models}, $G$ can compute a model $\A^*$ with $\a^*$-bf-type $\si^*$.
Since $\A$ and $\A^*$ are $\equiv_{\a^*}$-equivalent within $\M_{G,Y}$, they are, in particular, $\Sico{G\oplus Y}_{<\a}$-elementary equivalent.
This last part still holds in the universe because being $\Sico{G\oplus Y}_{\b}$-elementary equivalent for $\b<\a$ is a $\Delta^1_1$-property.
Since both $\A$ and $\A^*$ are computable in $G\oplus Y$, they are isomorphic by Theorem \ref{thm: effective Scott}. 

An identical argument, using a model $\M_{Z,G}$ containing $G$ and $Z$, proves that $Z$ computes a copy of $\A^*$, and hence of $\A$.
\end{proof}

\begin{theorem}\label{thm: minimal spectra ->}
Let $T$ be an $\L_{\om_1,\om}$ counterexample to Vaught's conjecture.
Then, there is an oracle $Z_0$ such that, for every set $\Cfr\subseteq 2^\om$, the following are equivalent:
\begin{enumerate}
\item $\Cfr$ is the degree spectrum of a model of $T$ relative to $Z_0$.   \label{part: minimal sp 1}
\item $\Cfr=\{Z\in 2^\om: \om_1^{Z\oplus Z_0}\geq \a\}$ for some admissible ordinal $\alpha$. \label{part: minimal sp 2}
\end{enumerate}
\end{theorem}

We note that saying that $\Cfr\subseteq 2^\om$ is the degree spectrum of $\A$ relative to $Z_0$ means that for every $Z$, $Z\in \Cfr$ if and only if $Z\oplus Z_0$ computes a copy of $\A$.
Saying that $\om_1^Z\geq\a$ relative to $Z_0$ means that $\om_1^{Z\oplus Z_0}\geq \a$.
So, the thesis of the theorem above states that $T$ satisfies (\ref{minimal main: part 2}) of Theorem \ref{thm: main} relative to $Z_0$.
It is easy to see that (\ref{minimal main: part 2})  implies  (\ref{main: part 2}).

\begin{proof}[Proof of Theorem \ref{thm: minimal spectra ->}]
We work relative to the oracle $Z_0$ given by Theorem \ref{thm: counter example}.
By increasing this oracle if necessary, we can also assume that $T$ is a computable infinitary sentence.
So, we have that, for every $Z$ and every model $\A\models T$,
\[
\Sp(\A) = \{Z\in 2^\om: \om_1^Z\geq \om_1^\A\}.
\]
This proves that (\ref{part: minimal sp 1}) implies (\ref{part: minimal sp 2}) for $\a=\om_1^\A$.
For the other direction, consider an admissible ordinal $\a$.
By Lemma \ref{lemma: Sacks 4.7}, there is a structure $\A\models T$ with $SR(\A)\geq \om_1^\A=\a$, and hence with spectrum $\{Z\in 2^\om: \om_1^Z\geq \a\}$.
\end{proof}

\subsection{Coding into the $\a$th-jump}  \label{se: coding}

Finally, we need to prove that  (\ref{main: part 2}) implies (\ref{main: part 1}).
The ideas in the proof below come from \cite[Theorem 3.1]{McountingBF}, but the setting is somewhat different.

Suppose $T$ satisfies hyperarithmetic-is-recursive relative to a cone.
Let us relativize the rest of the proof to the base of this cone.
Suppose, toward a contradiction, that $T$ has continuum many $\a$-bf-types for some $\a<\om_1$.
Let $\a$ be the least such, and, by relativizing again, let us assume that $\a$ is a computable ordinal.
Let us also assume that $\BFI{\a,0}$  has size continuum.
(This assumption is without loss of generality, as, if $\BFI{\a,k}$ has size continuum, we can add $k$ constants to $\L$ and then $\BFI{\a,0}$  has size continuum.)
Relativize again, if necessary, to make sure that $T$ has a computable $(<\a)$-bf structure.

Consider $\L_{<\a}=\bigcup_{\b<\a}\L_\b$.
As we mentioned in Section \ref{se: bf}, for every model $\A\models T$, $\A_{(<\a)}$ has a copy which is $\Delta^0_{\a}$ in the diagram of $\A$.
For each $\A$, let $t_\A$ be the finitary-$\Si_1$-theory of $\A_{(<\a)}$ (which determines the $\Sii_1$-$\L_{<\a}$-theory of $\A_{(<\a)}$, and hence the $\Sii_\a$-$\L$-theory of $\A$).
Note that every copy of $\A_{(<\a)}$ can enumerate $t_\A$.
In what follows, we view $t_\A$ as a member of $2^\om$, by considering an effective list of all the $\Si_1$-$\L_{<\a}$-sentences.

Let $R=\{t_\A: \A\in\KK\}\subseteq 2^\om$, where $\KK=\{\A:\A\models T\}$.
Since $\Sii_\a\tp_\A$ is determined by $t_\A$, we have that $t_\A=t_\B$ if and only if $\A\equiv_\a\B$.
Thus $R$ has size continuum.
Notice that $R\subseteq 2^\om$ is a $\mathbf{\Sigma}^1_1$ class, because $R$ is the image of $\KK$ under $t$, and $\KK$ and $t$ are Borel.
Since $R$ is uncountable and $\mathbf{\Sigma}_1^1$, Suslin's theorem (see \cite[Corollary 2C.3]{Mos80}) says that $R$ has a perfect closed subset $[S]$, determined by some perfect tree $S\subseteq 2^{<\om}$ (where $[S]$ is the set of paths through $S$).
In what follows, we relativize our construction to $S$, so we assume $S$ is computable.
Thinking of $S$ as an order-preserving map $2^{<\om}\to 2^{<\om}$,  for $X\in 2^\om$ we let $S(X)$ be the path through $S$ obtained as the image of $X$ under this map.

Now, let $Y$ be an oracle above all the relativizations we have done so far.
Let $X$ be hyperarithmetic in $Y$ and such that the set 
\[
Z=\{\si\in 2^{<\om}: X\geq_{lex} \si\}
\]
 is not $\Si^0_{\a}(Y)$.

$S(X)$ gives us a $\Si_1$-$\L_{<\a}$-type that is consistent with $T$ and of Turing degree $X$.
There is some $\A_{(<\a)}\in\KK$ with $\Si_1$-$\L_{<\a}$-type $t_\A=S(X)$.
Let 
\[
\hat{T} = T_\a \cup \{\psi: \psi\in S(X)\} \cup \{ \neg\psi: \psi \mbox{ a $\Si_1$-$\L_{<\a}$-sentence, } \psi\not\in S(X)\}. 
\]
Note that  $\hat{T}$ is a $\Pico{X}_2$ theory, and that the finite $\L_{<\a}$-substructures of the models of $\hat{T}$ are determined by $S(X)$.
Using Lemma \ref{lemma: building models}, we get a $X$-computable  model $\A_{(<\a)}\models\hat{T}$.
Let $\A$ be the $\L$-restriction of $\A_{(<\a)}$.
We have that $\A$ models $T$ and it is hyperarithmetic in $Y$; it remains to show that $\A$ has no copy computable in $Y$.
For any copy $\B$ of $\A$, we have $\B_{(<\a)}$ can enumerate $S(X)\subseteq\om$ and hence also the set $\{\si\in 2^{<\om}: S(X)\geq_{lex} \si\}$.
Since $S$ is computable, $\B_{(<\a)}$ can enumerate $Z$, and hence, $Z$ is $\Si^0_{\a}$ in the diagram of $\B$.
Since $Z$ was chosen not to be $\Si^0_{\a}(Y)$, we have that $\B$ cannot be $Y$-computable.


\end{document}